\documentclass[11pt,reqno]{article}

\usepackage{amsmath}
\usepackage{amssymb}
\usepackage{amsthm}
\usepackage{graphicx}
\usepackage{tikz}
\usepackage{xcolor,bbm,bm}

\usepackage{mathtools}
\mathtoolsset{showonlyrefs}

\allowdisplaybreaks

\usepackage[margin=1.25in]{geometry}

\usepackage{bookmark}

\def\E{\mathbb{E}}

\def\tran{\mathsf{T}}

\def\R{\mathbb{R}}
\def\N{\mathbb{N}}

\def\bx{\mathbf{x}}

\def\rsm{\bm{\sigma}}

\newcommand\dif{\mathop{}\!\mathrm{d}}

\def\pl{\mathrm{pl}}

\def\E{\mathbb{E}}
\def\R{\mathbb{R}}
\def\P{\mathbb{P}}

\def\eps{\varepsilon}

\def\1{\mathbf{1}}

\def\tce{t_c + \eps}
\def\tce2{t_c + \frac{\eps}{2}}

\newtheorem*{theorem*}{Theorem}
\newtheorem{theorem}{Theorem}
\newtheorem{lemma}[theorem]{Lemma}

\newtheorem{defn}[theorem]{Definition}
\newtheorem*{defn*}{Definition}

\newtheorem*{prop*}{Proposition}
\newtheorem{conjecture}[theorem]{Conjecture}
\newtheorem*{conj*}{Conjecture}

\newtheorem{question}[theorem]{Question}
\newtheorem{assumption}{Assumption}
\newtheorem{remark}{Remark}
\newtheorem*{fact*}{Fact}

\begin{document}
\title{Frozen $1$-RSB structure of the symmetric\\ Ising perceptron}

\author{Will Perkins\thanks{Department of Mathematics, Statistics, and Computer Science, University of Illinois at Chicago, math@willperkins.org.  Supported in part by NSF grants DMS-1847451 and CCF-1934915.} \and
Changji Xu \thanks{Center of Mathematical Sciences and Applications, Harvard University, cxu@cmsa.fas.harvard.edu.}}
\date{\today}

\maketitle
\thispagestyle{empty}

\begin{abstract}
We prove, under an assumption on the critical points of a real-valued function, that the symmetric Ising perceptron exhibits the `frozen 1-RSB' structure conjectured by Krauth and M{\'e}zard in the physics literature; that is,  typical solutions of the model lie in clusters of vanishing entropy density.  Moreover, we prove this in a very strong form conjectured by Huang, Wong, and Kabashima: a typical solution of the model is isolated with high probability and the Hamming distance to all other solutions is linear in the dimension.   The frozen 1-RSB scenario is part of a recent and intriguing explanation of the performance of learning algorithms by Baldassi,  Ingrosso,  Lucibello, Saglietti, and Zecchina.  We prove this structural result by comparing the symmetric Ising perceptron model to a planted model and proving a comparison result between the two models. Our main technical tool towards  this comparison is an inductive argument for the concentration of the logarithm of number of solutions in the model.  
\end{abstract}


\section{Introduction}

The \textit{perceptron model} is a simple model of a neural network storing random patterns.  It has been studied in several fields including information theory~\cite{cover1965geometrical}, statistical physics~\cite{G87,GD88,KM89,sompolinsky1990learning}, and probability theory~\cite{talagrand2010mean,kim1998covering,talagrand1999intersecting}. 

There are several variants of the model, grouped into two main categories: spherical perceptrons in which patterns are $N$-dimensional vectors on the unit sphere and Ising perceptrons in which patterns are $\pm 1$ vectors of length $N$.  In each case, we want to understand how many random patterns can be `stored' by a neural network formed by taking random synapses $J_{ij}$ and applying a given activation function.  Here we will study the Ising perceptron.

Let $\Sigma_N = \{\pm1 \}^N$ and let $\{X_i\}_{i \geq 1}$ be a sequence of independent $N$-dimensional standard Gaussian vectors\footnote{In another variant of the model the constraint vectors $X_i$ are given by independent samples from $\Sigma_N$ (Bernoulli disorder). While there are significant differences between the spherical and Ising perceptrons, the choice of Gaussian or Bernoulli disorder is insignificant for the properties discussed here.}.  For a real valued function $\phi$, a real number $\kappa$, and $X \in \R^N$, define 
\begin{equation}
H_{\phi, \kappa}(X) =  \left \{ \sigma \in \Sigma_N : \phi \left ( \langle X, \sigma \rangle /\sqrt{N}  \right) \le \kappa \right \} \,.
\end{equation}

  The \textit{solution space} of the Ising perceptron with Gaussian disorder, activation function $\phi$, threshold $\kappa$, and $m$ constraints is the random subset of $\Sigma_N$, 
\[ S= S_{\phi,\kappa, N, m} =  \bigcap _{i=1}^m   H_{\phi,\kappa}(X_i)   \,.\]
Thus $S$ is a random subset of the Hamming cube $\Sigma_N$.  We call the vectors $X_i$ \textit{constraint vectors}. The constraints depend on the set of constraint vectors, the activation function $\phi$, and the threshold $\kappa$. 
The classic Ising perceptron corresponds to the choice $\phi(x) = x$ where the most studied case is $\kappa= 0$ (e.g.~\cite{KM89,ding2019capacity}).

\subsection{Structure of the solution space}
\label{secStructure}

We will be concerned with the typical structure of the solution space  $S$ as a function of $\phi, \kappa$, and the constraint density $\alpha := m/N$, as $N \to \infty$. 

The most basic structural question is whether $S$ is empty or not. The \textit{capacity} of the perceptron is defined as the random variable
\begin{equation}
M_{\phi, \kappa}(N) =  \max \{ m :  S_{\phi,\kappa, N, m} \ne \emptyset \} \,, 
\end{equation}
and the \textit{critical capacity density}
\begin{equation}
\alpha_{c,\phi} (\kappa) = \inf \left \{ \alpha :  \liminf_{N \to \infty} \Pr[  S_{\phi,\kappa, N, \lfloor \alpha N \rfloor} =\emptyset ] =1  \right \} \, 
\end{equation}
is the typical constraint density of the capacity.   

For densities below the critical capacity density, when  the solution space is typically non-empty, we can ask about its  structure, and how this structure affects the performance of learning algorithms: algorithms that find some solution in $S$ given the instance defined by $\mathbf  X = (X_1, \dots, X_m)$.   

  Basic structural questions include whether solutions appear in connected clusters or are isolated; and what  the typical distance is from a solution to the next nearest solution.  For these structural properties we regard $\Sigma_N$ as the Hamming cube endowed with Hamming distance: for $\sigma, \sigma' \in \Sigma_N$,  
$$\mathrm{dist}(\sigma,\sigma') = |\{ i: \sigma_i \ne \sigma'_i \}| = \frac{n-  \langle \sigma, \sigma' \rangle }{2} \, .$$   

The perspective taken in recent work on Ising perceptrons in both statistical physics and computer science is to view the Ising perceptron as a random constraint satisfaction problem (CSP): $\Sigma_N$ is the set of possible solutions, and each random vector $X_i$ defines a constraint $ \phi \left ( \langle X, \sigma \rangle /\sqrt{N}  \right) \le \kappa$ on possible solutions $\sigma \in \Sigma_N$.  The critical capacity density $\alpha_{c,\phi}$ is then the satisfiability threshold of the model.  

Just as in the random $k$-SAT, random $k$-NAE-SAT, or random $k$-XOR-SAT problems, each constraint rules out a constant fraction of all solutions in $\Sigma_N$.  Where the perceptron differs from these other models is that in the perceptron each constraint involves all of the $N$ coordinates, while in the other models each constraint only involves $k$, which is held constant as $N \to \infty$. 

Random CSP's have been studied extensively in computer science, statistical physics, probability, and combinatorics since Mitchell, Selman, and Levesque~\cite{mitchell1992hard} observed empirically that random $k$-SAT formulae at certain densities proved extremely challenging for widely used SAT solvers.  Understanding this phenomenon is a major ongoing challenge that has led to the development of a new field of inquiry at the intersection of computer science, physics, and mathematics. 

A key to the current understanding of random CSP's is understanding the typical structure of the solution space at different constraint densities.  A beautifully detailed but non-rigorous picture was put forth by Krzaka{\l}a, Montanari,  Ricci-Tersenghi, Semerjian, and Zdeborov{\'a}~\cite{krzakala2007gibbs}, based on the cavity method from statistical physics. 

In a paper that transformed the computer science perspective on random computational problems, Achlioptas and Coja-Oghlan~\cite{achlioptas2008algorithmic}  proved the existence of a `clustering threshold' at which the solution space of certain random CSP's breaks apart into exponentially many clusters separated by linear Hamming distance. They also observed that this clustering threshold coincides asymptotically with the threshold at which known efficient search algorithms for these problems fail.  

A further structural property is \textit{freezing}.  Given a solution $\sigma$, a variable (or coordinate) is free if flipping that coordinate results in another solution $\sigma'$.  Variables that are not free are {frozen}.  The freezing threshold for a random CSP is the threshold at which typical solutions have a linear number of frozen variables~\cite{zdeborova2007phase,molloy2018freezing}. 

These two properties -- clustering and freezing -- are conjectured to be the source of computational hardness in random CSP's.  Identifying  thresholds for the onset of these and other structural phenomena and rigorously connecting them to the performance of algorithms have been a main focus of the field of random computational problems in the last decade.

For the Ising perceptron, the conjectured structural picture looks strikingly different.  Krauth and Mezard~\cite{KM89} conjectured, by means of the replica method, that at all densities below the critical capacity, the solution space of the perceptron is dominated by clusters of vanishing entropy density (that is, of size $e^{o(n)}$), each cluster separated from the others by linear Hamming distance.  Wong, Kabashima, and Huang~\cite{huang2013entropy} and Huang and Kabashima~\cite{huang2014origin} refined these conjectures and posited that in fact typical solution in the Ising perceptron are  \textit{completely frozen}; that is, all coordinates are frozen and the solutions lie in clusters of size $1$, separated from all other solutions by linear Hamming distance.   Thus the Ising perceptron exhibits clustering and freezing in the strongest possible form throughout the entire satisfiability regime.    Based on the current conjectural understanding of random CSP's (see~e.g.~\cite{zdeborova2008constraint}), one might venture a guess that finding a solution in the Ising perceptron is computationally hard at all positive densities.  

However,  this theory is at odds with other work in physics on learning algorithms for the perceptron.  Braunstein and Zecchina~\cite{braunstein2006learning} observed empirically that a simple message-passing algorithm is able to find solutions at positive densities in the Ising perceptron (further algorithms followed in~\cite{baldassi2007efficient,baldassi2009generalization}).  Attempting to reconcile this apparent contradiction between the theory and empirical observations, Baldassi, Ingrosso, Lucibello,  Saglietti, and Zecchina~\cite{baldassi2015subdominant} conjectured that these successful learning algorithms were in fact finding solutions belonging to rare clusters of positive entropy density.  That is, although a $1-o(1)$ fraction of solutions belong to isolated, frozen clusters, an exponentially small fractions of solutions belong to clusters that are exponentially large; strikingly, the authors observed that learning algorithms find solutions in these rare clusters.  Specifically, the solutions that contribute to  the dominant portion of the partition function (number of solutions) and determine the equilibrium properties of the model are completely distinct from those solutions that efficient algorithms find.   This work was followed by the proposal of several different algorithms to target these subdominant clusters~\cite{baldassi2016local,baldassi2016unreasonable}.  

 In~\cite{aubin2019storage}, a \textit{symmetric Ising perceptron}, with activation function $\phi(x) = |x|$, was studied as a model conjectured to exhibit the same structural and algorithmic  properties, but more amenable to mathematical analysis.  Baldassi, Della Vecchia, Lucibello, and Zecchina~\cite{baldassi2019clustering} confirmed that on the level of the physics predictions this model has the same qualitative behavior as the  Ising perceptron with activation function $\phi(x) =x$.

In summary, the Ising perceptron, and  its symmetric variants, are conjectured to exhibit `frozen $1$-RSB behavior' at all positive densities below the critical capacity density.  The current understanding of the link between clustering, freezing, and the performance of algorithms would suggest that finding a solution in these models is therefore intractable . This, however, is seemingly in contradiction with empirical observations, and one hypothesis suggests that learning algorithms find exponentially rare solutions with atypical structural properties. 

Resolving these questions is a pressing problem since the hypothesis about subdominant clusters calls into question the link between the equilibrium properties of these models and algorithmic tractability.  In this work, we take a first step in addressing this problem rigorously  by establishing the frozen $1$-RSB picture for the symmetric Ising perceptron (Theorem~\ref{lem:strong-clustering} below), under an assumption on the critical points of a real-valued function (Assumption~\ref{assumption1}).  

\subsection{Previous results}
There are few rigorous results on the Ising perceptron, and most are concerned with bounds on the critical capacity.

For the classic Ising perceptron Krauth and M{\'e}zard~\cite{KM89} predicted, using the replica method, that $\alpha_c(\kappa)  = \alpha_{\mathrm{KM}}(\kappa)$ for a complicated but explicit function $\alpha_{\mathrm{KM}}$ (with $\alpha_{\mathrm{KM}}(0) \approx.83$).  Following some previous bounds of Kim and Roche and Talagrand~\cite{kim1998covering,talagrand1999intersecting}, Ding and Sun~\cite{ding2019capacity} recently proved that $\alpha_c (\kappa) \ge \alpha_{\mathrm{KM}}(\kappa)$ using a sophisticated form of the second-moment method  guided by the Thouless--Anderson--Palmer (TAP) equations~\cite{thouless1977solution}.   Their result assumes a technical condition on a certain real-valued function,  akin to Assumption~\ref{assumption1} below. 

Much of the technical difficulty of~\cite{ding2019capacity}  comes from the asymmetry inherent in the activation function $\phi(x) =x$; this  necessitates a conditioning argument and the sophisticated second-moment calculation.  On the other hand, Aubin, Perkins, and Zdeborov{\'a}~\cite{aubin2019storage} considered two symmetric activation functions: $\phi_r(x) = |x|$ and $\phi_u(x) = - |x|$, which they called the rectangular and `u' activation functions respectively.  Studying symmetric constraints has a long history in the random CSP literature: the random $k$-NAE-SAT model is a symmetric variant of the random $k$-SAT model.  While the qualitative properties of the two models are expected to be very similar, the symmetric model is often more amenable to rigorous analysis, and thus a clearer understanding can be obtained  (see e.g.,~\cite{ding2014satisfiability,sly2016number,bartha2019breaking} for recent work on the $k$-NAE-SAT model).  Studying symmetric perceptrons allows us to prove stronger and more detailed results than are currently attainable for the classic perceptron, but the phenomena studied are expected to be universal.  

Aubin, Perkins, and Zdeborov{\'a}~\cite{aubin2019storage} determine the critical capacity density for the  symmetric perceptron with rectangular activation function:
\begin{equation}
\alpha_{r,c} (\kappa) =  -\frac{\log 2}{\log p(\kappa)}
\end{equation}
where $p(\kappa) = \Pr [ |Z | \le \kappa]$ for a standard Gaussian random variable $Z$.  This result, like that of~\cite{ding2019capacity}, is contingent on an assumption about a certain real-valued function.  This function will also prove useful in our work.  Let  $H(\beta) = - \beta \log \beta - (1 - \beta) \log (1 - \beta)$ be the Shannon entropy function (all logarithms in this paper are base $e$) and let 
\begin{equation}
\label{eqQk}
q_{\kappa}(\beta) = \Pr( |Z_1| \le \kappa, |Z_2| \le \kappa) 
\end{equation}
 where $(Z_1,Z_2)$ is a jointly Gaussian vector with means $0$, variances $1$, and  covariance $2 \beta-1$. 
\begin{assumption}[\cite{aubin2019storage}]
\label{assumption1}
The function 
\begin{equation}
\label{eqFdef}
F_{\alpha}(\beta) = H(\beta) + \alpha \log q_{\kappa}(\beta) \, ,
\end{equation}
 has a single critical point for $\beta \in (0,1/2)$ whenever $F_{\alpha}''(1/2)<0$. 
\end{assumption}

\begin{remark}
We (and the authors of~\cite{aubin2019storage}) have plotted $F_{\alpha}(\beta)$ for many choices of $\alpha$ and $\kappa$ and in all instances the function has the shape depicted in Figure~\ref{figure1}, consistent with Assumption~\ref{assumption1}.  We believe a proof of the assumption might be possible adapting the methods  of~\cite[Proof of Lemma 3]{achlioptas2002asymptotic}.
\end{remark}

  Xu~\cite{xu2019sharp}  proved a general sharp threshold result (an analogue of Friedgut's sharp threshold result for random graphs and CSP's~\cite{friedgut1999sharp}), which, combined with~\cite{aubin2019storage}, gives a sharp threshold for the existence of solutions: for any $\eps>0$, 
\begin{align*}
\Pr[ S_{\phi_r, \kappa, N (\alpha_{r,c} +\eps)N} = \emptyset] &\to 1  \text{ as  } N \to \infty \\
\Pr[ S_{\phi_r, \kappa, N (\alpha_{r,c} -\eps)N} = \emptyset] &\to 0  \text{ as  } N \to \infty  \,.
\end{align*}
Both statements hold also for the $u$-function in a range of $\kappa$ values, and the results of~\cite{ding2019capacity,xu2019sharp} prove that the second statement holds for the classic perceptron.   Proving the matching  upper bound on the critical capacity for the classic perceptron remains a challenging open problem.

Baldassi, Della Vecchia, Lucibello, and Zecchina~\cite{baldassi2019clustering} used the second-moment method to show the existence of pairs of solutions at arbitrary distances in the symmetric Ising perceptron.

\subsection{Main results}

We will study properties of typical solutions in the {symmetric Ising perceptron} (with the rectangular activation function $\phi(x) = |x|$).   We now specialize and simplify the  notation from above.

For  $X \in \R^N$ and $\kappa > 0$, define
\begin{equation}
	H_{\kappa}(X) := \{\sigma \in \Sigma_N: |\langle X,\sigma \rangle| \leq \kappa \sqrt{N}\}\,.
\end{equation}	
Let $\{X_i\}_{i \geq 1}$ be a sequence of i.i.d.\ $N$-dimensional standard Gaussians, and define the solution space
\begin{equation}
\label{eq:def-S-alpha}
	S = S_{\alpha} (N)  = \bigcap_{i = 1}^{\lfloor \alpha N \rfloor}H_{\kappa}(X_i)\,.
\end{equation}
The critical capacity density, determined in~\cite{aubin2019storage}, is $\alpha_c = \alpha_c(\kappa)= -\log2/\log p(\kappa)$.

The main result of this paper confirms the frozen $1$-RSB scenario in the symmetric Ising perceptron: typical solutions are completely frozen with high probability for $\alpha< \alpha_c$.    
Let
 \begin{equation}
\label{eq:def-dc}
\beta_c = \beta_c(\kappa, \alpha) =  \beta \in (0,1/2) : F(\beta) - \alpha \log p(\kappa) =0  \,.
\end{equation}
See Figure~\ref{figure1} for a depiction of $\beta_c$. We show in Lemma~\ref{lemDc} that for $\alpha < \alpha_c$, $\beta_c >0$ exists and is uniquely defined.

\begin{figure}[htbp]
\begin{center}
\includegraphics[scale=.45]{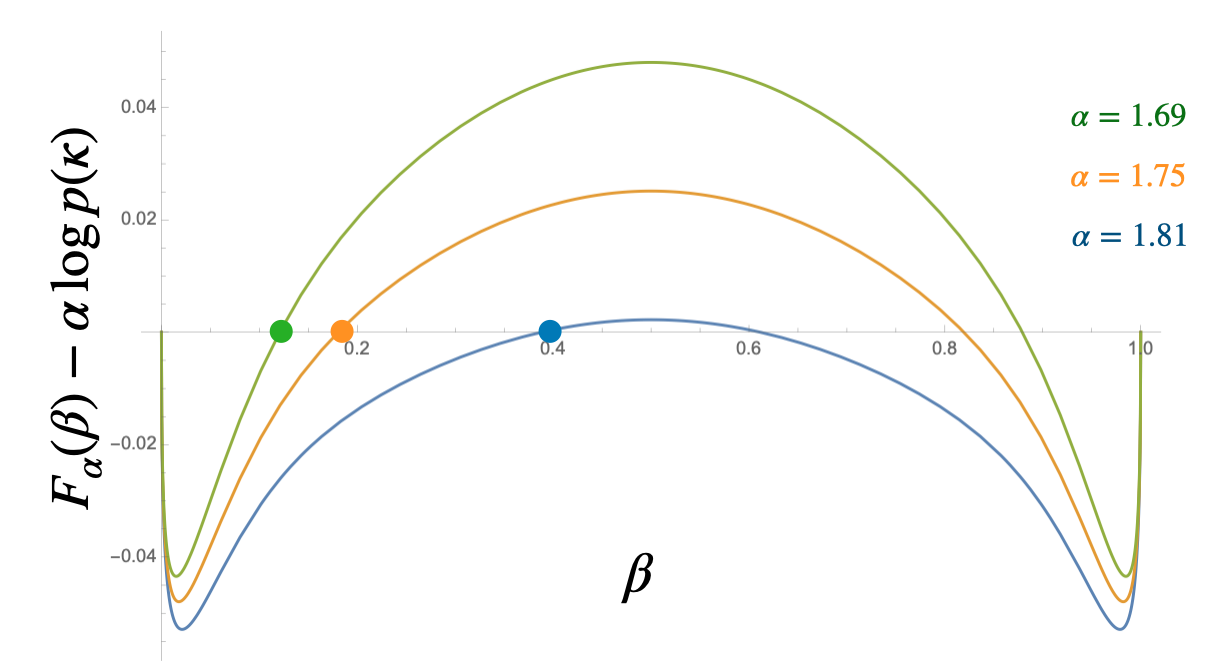}
\caption{$F_{\alpha}(\beta) - \alpha p(\kappa)$ plotted for $\alpha = \{ 1.69, 1.75, 1.81\}$  and $\kappa=1$. The dots mark $\beta_c$ for the three values of $\alpha$.}
\label{figure1}
\end{center}
\end{figure}

\begin{theorem}
\label{lem:strong-clustering}
Let $\kappa >0$ and  $\alpha< \alpha_c(\kappa)$.  Let $\sigma$ be uniformly sampled from $S$ conditioned on the event $S \ne \emptyset$.  Under Assumption~\ref{assumption1}, for any $\delta \in (0,\beta_c)$, 
	\begin{equation}
		\{ \sigma' \in S : \mathrm{dist} (\sigma, \sigma' ) \leq ( \beta_c - \delta)N \}  = \{\sigma\}\,.
	\end{equation}
	with probability $1-o(1)$ as as $N \to \infty$.
	In particular, $\sigma$ is completely frozen with probability $1-o(1)$. 
\end{theorem}
Note that $\sigma$ is selected according to two sources of randomness: the randomness of the perceptron instance $\mathbf X$ and the random choice of $\sigma$ from $S$.  

The next result shows that the logarithm of the number of number of solutions in the rectangular Ising perceptron is tightly concentrated below the critical density. 
\begin{theorem}
\label{PartitionF}
Under Assumption~\ref{assumption1}, for $\alpha< \alpha_c$
	\begin{equation}
		\frac{\log |S|}{N}   =  \log 2+ \alpha \log  p(\kappa)  + O_{\P}\Big(\frac{\log N}{N}\Big)\quad \text{as } N\to \infty\text{~in~probability}\,. 
	\end{equation}
	 In particular, for $\alpha < \alpha_c$, $S$ is non-empty with probability $1-o(1)$. 
\end{theorem}
The second statement of the Theorem~\ref{PartitionF} proves that the symmetric Ising perceptron undergoes a sharp satisfiability phase transition at $\alpha_c $, answering an open question from~\cite{aubin2019storage} (where the complementary statement that for $\alpha > \alpha_c$, $S = \emptyset$ with high probability is proved).  This could also be proved by adapting the sharp threshold result of~\cite{xu2019sharp} to the symmetric perceptron.

\subsection{Overview of the techniques}

We study the properties of a typical solution drawn from $S$ by way of the \textit{planted model}: the experiment of first  selecting a uniformly random solution from $\Sigma_N$, then choosing a random configuration of constraints consistent with this solution.    Planted models have been studied extensively in the random CSP literature and beyond.   They are used as a toy model for statistical inference: e.g. the `teacher--student model'~\cite{zdeborova2016statistical} or the stochastic block model~\cite{abbe2017community}.  They are used to understand the condensation threshold in random CSP's~\cite{coja2012condensation,bapst2016condensation,coja2018information,coja2018charting}.  They are used to understand the structure of the solution space in random CSP's~\cite{achlioptas2008algorithmic,achlioptas2011solution,montanari2011reconstruction,molloy2018freezing}.

Given $N, m \in \mathbb N$, $\kappa >0$, and following~\cite{achlioptas2008algorithmic}, we define two probability distributions on pairs $(\sigma^*, \mathbf X) \in \Sigma_N \times (\R^N)^m$ of solutions and configurations of  $m$ constraint vectors.  

\medskip

In the \textbf{random model} we:
\begin{enumerate}
\item Sample $m$ i.i.d.\ $N$-dimensional standard Gaussian constraint vectors $\mathbf X = (X_1, \dots , X_m)$, conditioned on the event that $S (\mathbf X) = \bigcap_{i=1}^m H_{\kappa}(X_i) \ne \emptyset$.
\item Sample $\sigma^*$ uniformly at random from $S$. 
\end{enumerate}
We denote the law of the random model with $\P_{\mathrm{r}}, \E_{\mathrm{r}}$ to distinguish the law from both the unconditional perceptron model and the planted model below.   The random model is simply the experiment of selecting a uniformly random solution from the symmetric Ising perceptron conditioned on satisfiability.  

\medskip

In the \textbf{planted model} we:
\begin{enumerate}
\item Sample $\sigma^*$ uniformly at random from $\Sigma_N$. 
\item Sample a configuration of  $m$ i.i.d.\ constraint vectors $\mathbf X = (X_1, \dots X_m)$, with each $X_i$ distributed as a standard $N$-dimensional Gaussian vector conditioned on the event  that $\sigma^* \in H_{\kappa}(X_{i})$.
\end{enumerate}
We denote the law of the planted model with $\P_{\mathrm{pl}}, \E_{\mathrm{pl}}$.

The key to using the planted model to understand the original model is to show that at low enough constraint densities, the two distributions on $(\sigma^*, \mathbf X)$ are close.  Proving that the distributions are close,  as we do below in Lemma~\ref{contiguity},  amounts to proving that the number of solutions, $|S|$, is typically not too far from its expectation, $\E |S|$.  The better concentration of $|S|$ one can prove, the more one can deduce about the original model from the planted model.   In~\cite{achlioptas2008algorithmic} it is shown (in the case of $q$-colorings of a random graph) that if $\log |S| = \log \E |S| + o(N)$, then events that occur with probability at most $\exp( - \theta(N))$ in the planted model occur with probability $o(1)$ as $N \to \infty$ in the random model.   This notion of closeness is `quiet planting'~\cite{krzakala2009hiding} and it suffices to prove some structural results on the solution spaces such as clustering~\cite{achlioptas2008algorithmic}.  On the other hand, much stronger notions of closeness have been proved: `silent planting'~\cite{bapst2017planting} which implies the two distributions are mutually  contiguous: any event with probability $o(1)$ in the planted model has probability $o(1)$ in the random model.  This has been used to prove stronger structural results~\cite{molloy2018freezing}.  Proving contiguity requires much stronger concentration of $\log |S|$.  A very general result on the contiguity of the planted and random model for symmetric CSP's~\cite{coja2018charting} involves a rigorous implementation of the cavity method and the small subgraph conditioning method.  In the setting of the perceptron, neither of these tools exist and so we must prove concetration via another route.  

We  prove Theorem~\ref{lem:strong-clustering} in three steps.

In Section~\ref{secPlanted}, we prove that the planted solution is isolated and the next nearest solution is at linear Hamming distance with high probability in the planted model (Lemma~\ref{pl-cluster}).   

In Section~\ref{secConcentrate} we prove Theorem~\ref{PartitionF}, showing that for $\alpha< \alpha_c$, the logarithm of the number of solutions in the random model is concentrated around the logarithm of the expected number of solutions.  To the best of our knowledge, our approach to proving concentration in this way is new, and we expect the technique to have further applications.

In Section~\ref{secQuiet} we transfer our results about the planted model to the random model by showing that events that occur with probability at most $N^{-\omega(1)}$ in the planted model occur with probability $o(1)$ in the standard model (Lemma~\ref{contiguity}).  This relies on the concentration properties of the logarithm of the number of solutions.

\subsection{Extensions and future work}
Both Theorem~\ref{lem:strong-clustering} and Theorem~\ref{PartitionF} can be extended verbatim to the $u$-function Ising perceptron studied in~\cite{aubin2019storage}, for $\kappa \in (0, .817) $ (the same range of $\kappa$ for which the second-moment method  works there).

The main open problem in this area is to resolve the conceptual dilemma  described in Section~\ref{secStructure} and answer the questions raised in~\cite{baldassi2015subdominant}.  Are there efficient learning algorithms that always find out-of-equilibrium solutions in subdominant clusters?  This would raise a serious challenge to the belief that an understanding of the associated equilibrium statistical mechanics model (on the level of the free energy) can explain computational tractability or intractability.  

  Concretely, now that we have verified the frozen $1$-RSB scenario, we can ask for provably efficient learning algorithms for the symmetric Ising perceptron.
\begin{question}
Is there a polynomial-time algorithm that, with probability $1-o(1)$,   finds a solution to the symmetric Ising perceptron for some  density $\alpha\in (0, \alpha_c)$?
\end{question}

Alternatively, one could leverage the structure results we have proved here and rule out some class of learning algorithms.

\medskip

We leave the following additional open problems for future work.

\medskip

1.  {Prove that the classic perceptron with activation function $\phi(x)= x$ exhibits the frozen $1$-RSB property.}    It is not at all clear how to extend the method of this paper to this case.  As discussed in~\cite[Sec.\ 2.4]{coja2018charting}, for asymmetric random CSP's, like the random $k$-SAT model, the planted model, at least in its straightforward implementation, is not useful to compare to the random model (in particular it is not contiguous with the random model at any positive density).  We expect the same with the classic perceptron, and so our strategy of arguing via the planted model will not work.

2. {Prove full contiguity between the random and planted models for $\alpha<\alpha_c$.}  Our comparison result (Lemma~\ref{contiguity}) suffices for our purposes here, but it is natural to ask for more (as is the case for a large class of symmetric random CSP's~\cite{coja2018charting}).
\begin{conjecture}
For $\alpha< \alpha_c$, the random and planted models of the symmetric Ising perceptron are mutually contiguous.  That is,  if $\P_{\pl}(A) = o(1)$ then $\P_{\mathrm{r}}(A) = o(1)$ and vice versa. 
\end{conjecture}

\section{The planted model}
\label{secPlanted}

Consider the planted model with planted solution $\sigma^*$ and  constraints  $\mathbf X =(X_1, X_2, \dots, X_m) $.  Define $S (\mathbf X) = \bigcap_{i \geq 1}^{\lfloor \alpha N \rfloor}H_{\kappa}(X_{i})$ as in \eqref{eq:def-S-alpha}.  Recall the definition of $\beta_c$ from~\eqref{eq:def-dc}. 
We  show that $\beta_c$ exists and is unique for $\alpha \in (0, \alpha_c)$.
\begin{lemma}
\label{lemDc}
Under Assumption~\ref{assumption1}, for $\alpha \in (0, \alpha_c)$, there exists a unique $\beta \in (0,1/2)$ so that 
\[ F_{\alpha}(\beta) - \alpha \log p(\kappa) = 0 \,. \]
Also, for any $\delta \in (0,\beta_c/2)$,
\begin{equation}
\label{eq:G<0}
	\sup_{\delta < \beta < \beta_c - \delta}F_{\alpha}(\beta) - \alpha \log p(\kappa) < 0\,.
\end{equation}
\end{lemma}
\begin{proof}
Let $G(\beta) = F_{\alpha}(\beta) - \alpha \log p(\kappa)$.  Then since $q_\kappa(0) = p(\kappa)$ and $H(0) =0$, we have $G(0)=0$.  As observed in~\cite{aubin2019storage}, $G'(0) = - \infty$ and so for some $\eps>0$, $G(x) < 0$ for $x \in (0, \eps)$.  Moreover, since $\alpha< \alpha_c$, $G(1/2) >0$, and so by continuity there exists $\beta \in (0,1/2)$ with $G(\beta)=0$.  By Assumption~\ref{assumption1} this $\beta$ is unique. In addition, Assumption~\ref{assumption1} implies that $G(\beta)$ is first strictly decreasing and then strictly increasing on $\beta\in(0,1/2)$. This gives  \eqref{eq:G<0}.
\end{proof}

  The following result says the planted solution is completely frozen with high probability in the planted model.
\begin{lemma}
\label{pl-cluster}
Under Assumption~\ref{assumption1}, for any $\alpha \in (0, \alpha_c)$ and any $\delta \in (0, \beta_c)$, there exists a constant $c_\delta>0$ such that
				\begin{equation}
		\P_{\pl}\Big(\{ \sigma \in S: \mathrm{dist} ( \sigma, \sigma^*)    \leq (\beta_{c} - \delta)N \} \not = \{\sigma^*\}\Big) \leq \exp\left\{- c_{\delta}\sqrt{N}\right\}\,.
	\end{equation}
\end{lemma}

Before proving Lemma~\ref{pl-cluster} we introduce some notation. 
Let $q(m) = \P(\sigma,\sigma' \in H_{\kappa}( X))$, where $\sigma,\sigma' \in \Sigma_N$ are two arbitrary vectors with  $\langle \sigma,\sigma' \rangle = m$ and $X$ is a standard $N$-dimensional Gaussian vector.  Then
\begin{equation}
\label{eq:qmquk}
	q(m) =  q_{\kappa}\Big(\frac{1}{2} + \frac{m}{2N}\Big)\,,
\end{equation}
where $q_\kappa$ is defined in~\eqref{eqQk}.

Lemma~\ref{pl-cluster} will follow from the following results.
\begin{lemma}
\label{Boundary}
There exists a constant $\epsilon>0$ sufficiently small such that for all $m \leq \epsilon N$,
	\begin{equation}
		 \log \E_{\pl}[|\{\sigma \in S: |\langle \sigma, \sigma^*\rangle| = N-m\}|] \leq -c \sqrt{mN}\,.
	\end{equation}
\end{lemma}
\begin{proof}
	Let $\sigma_0$ be an  vector in $\Sigma_N$. Note that
	\begin{equation}
	\label{eq:1-P-edge}
	\begin{split}
				\E_{\mathrm{pl}}[|\{\sigma \in S: \langle \sigma, \sigma^*\rangle = N-m\}|]  &= \sum_{\sigma: \langle \sigma, \sigma_0\rangle = N-m} \P(\sigma \in S \mid \sigma_0 \in S)\\
		& = \binom{N}{\frac{m}{2}} \Big(\frac{q(m)}{p(\kappa)}\Big)^{\alpha N}\,.
	\end{split}
	\end{equation}
We claim that uniformly over all $|m| \leq \epsilon N$ with sufficiently small $\epsilon$,
\begin{equation}
\label{eq:2.5}
	 q(N-m) \leq p(\kappa) - c \sqrt{m/N}
\end{equation}
Provided with \eqref{eq:2.5}, we have
\begin{equation}
	\eqref{eq:1-P-edge} \leq \exp\Big\{ m\log(2N/m)/2 - c\sqrt{mN}\Big\} \leq \exp\Big\{-c\sqrt{mN}\Big\}\,.
\end{equation}
Hence Lemma \ref{Boundary} follows. Now it remains to prove \eqref{eq:2.5}. To this end, let $\mathbf 1$ be the all $1$'s vector of length $N$ and $\mathbf {1}_m$ be an $N$-dimensional vector with the first $(N-m)$ coordinates  $+1$ and the remaining $m$ coordinates  $-1$.  We write
\begin{align*}
		p(\kappa) - q(N-m)  & = \P(|\langle \mathbf {1},X\rangle| \leq \kappa \sqrt{N} , |\langle  {\mathbf 1_{\tfrac m 2}, X}\rangle| > \kappa \sqrt{N})\\
		& \geq \P\left(\sum_{i=1}^{N-\tfrac m 2} X_i \in (\kappa \sqrt{N},\kappa \sqrt{N} + \sqrt{m}), \sum_{i = N-\tfrac m 2+1}^N X_i \in (-3\sqrt{m},-2\sqrt{m})\right)\\
		& = \P\left(Z_1 \in \left(\kappa \sqrt{\frac{N}{N-\tfrac m 2}},\kappa \sqrt{\frac{N}{N-\tfrac m 2}} + \sqrt{\frac{m}{N-\tfrac m 2}}\right), Z_2 \in (-3\sqrt 2,-2 \sqrt 2)\right)\\
		& \geq c\sqrt{m/N}
	\end{align*}
	where $Z_1,Z_2$ are two independent standard Gaussian random variables.
This proves \eqref{eq:2.5}.
\end{proof}

\begin{lemma}
\label{1-dis-con}
For any $\epsilon>0$, uniformly in $ |m| \leq N - \epsilon N$,
	\begin{equation}
	\label{eq:1-dis-con}
		\lim_{N\to \infty}\Big|\frac{1}{N}\log \E_{\pl}\big[\big|\{\sigma \in S: \langle \sigma, \sigma^*\rangle = m\}\big|\big] - \Big[F\Big(\frac{1}{2} + \frac{m}{2N}\Big) - \alpha \log p(\kappa) \Big]\Big| =0\,.
	\end{equation}
\end{lemma}

\begin{proof}
	Note that
	\begin{equation}
		\E_{\pl}\big[\big|\{\sigma \in S: \langle \sigma, \sigma^*\rangle = m\}\big|\big] = \binom{N}{\frac{N+m}{2}} \cdot \left(\frac{q(m)}{p(\kappa)}\right)^{\alpha N}
	\end{equation}
	and that for any $\epsilon>0$, uniformly in $\epsilon N \leq k \leq N - \epsilon N$,
	\begin{equation}
		\frac{\frac{1}{N}\log \binom{N}{k}}{H(k/N) } \to 1 \quad \text{ as $N \to \infty$.}
	\end{equation}
Combining these with \eqref{eq:qmquk} yields \eqref{eq:1-dis-con}.
\end{proof}
\begin{proof}[\bf Proof of Lemma \ref{pl-cluster}]Let $\delta \in(0,\beta_c)$. It suffices to show that
	\begin{equation}
	\label{eq:pl-ex}
			\E_{\pl}|\{\sigma \in S : 0  < \mathrm{dist}( \sigma, \sigma^*)  \leq (\beta_c-\delta)N \}| \leq \exp\left\{- c_{\delta}\sqrt{N}\right\}\,.
\end{equation}
This bound follows from Lemmas \ref{Boundary} and \ref{1-dis-con} and \eqref{eq:G<0}.
\end{proof}

\section{Concentration of the number of solutions}
\label{secConcentrate}
Fixing $N$, consider the symmetric perceptron as a discrete-time stochastic process with one constraint vector added at each time step.
The solution space at time $t \in \mathbb N$ is defined as
\begin{equation}
	S_{t} := \bigcap_{i=1}^{t} H_\kappa(X_i)
\end{equation}
which is the intersection of $t$ random rectangles. 

The following  strengthening of Theorem~\ref{PartitionF} is the main result of this section.  
\begin{theorem}
\label{Stub}
Under Assumption \ref{assumption1}, for every $\epsilon>0$ there exists $M = M(\epsilon)$ such that for any $\alpha < \alpha_c$,
	\begin{equation}
	\label{eq:Stub}
		\limsup_{N \to \infty}\sup_{0 \leq t \leq \alpha N}\P\left(\Big|\log \Big(\frac{|S_t|}{\E[|S_t|]} \Big) \Big|\geq M \log N\right) \leq \epsilon\,.
			\end{equation}
\end{theorem}
Theorem~\ref{PartitionF} follows immediately since $\frac{1}{N} \log \E |S_t| = \log 2 + \frac{t}{n} \log p(\kappa)   $.

Theorem~\ref{Stub} says that the cardinality of the solution space will only deviate from its expectation slightly after adding $\alpha N$ random constraints. To prove this theorem, we will look at the change in this deviation at each time when a new constraint is added.  Write
\begin{equation}
	Q_t := \log \left( \frac{|S_t|}{\E[|S_t|]}\right) =\sum_{i = 1}^t\log \left( \frac{|S_i|/|S_{i-1}|}{\E[|S_i|]/\E[|S_{i-1}|]}\right)\,.
\end{equation}
Note that
$\E[|S_i|]/\E[|S_{i-1}|] = p(\kappa)$. Let
\begin{equation}
	Y_{t} :=\frac{1}{p(\kappa)}\Big(\frac{|S_{t}|}{|S_{t-1}|} -p(\kappa)\Big)\,,
\end{equation}
so that
\begin{equation}
\label{eq:QtXt}
	Q_t = \sum_{i = 1}^t \log (1 + Y_i) \,.
\end{equation}
Since $0 \leq |S_t| \leq |S_{t-1}|$, we have that $-1 \leq Y_t \leq (1 - p(\kappa))/p(\kappa)$; however, we expect that the $Y_t$'s are very close to zero with high probability. Hence by a Taylor expansion, as $N \to \infty$,
\begin{equation}
	Q_{t} = \sum_{i = 1}^{t}  Y_i - \frac{Y_i^2}{2} + o (1)\,.
\end{equation}
 In fact, we will prove in Lemma \ref{Xttail} that $Y_t$ is roughly of order $N^{-1/2}$ provided that $S_{t-1}$ is ``regular'' (see Definition \ref{def-reg}). Thus if the $S_{t}$'s are all regular, we expect the second term $ \sum_{i = 1}^{\alpha N}Y_i^2$ be of order $O(1)$. In addition, notice that $(\sum_{i = 0}^t Y_i)_{t \geq 0}$ is a martingale with respect to the filtration
\begin{equation}
 	\mathcal F_t = \sigma(S_1,S_2,...,S_t), \quad t \geq 1\,.
 \end{equation}
Hence if the $S_{t}$'s are all regular, we also expect the first term $\sum_{i = 1}^{\alpha N}  Y_i $ to be of order $O(1)$, and hence $Q_t = O(1)$.

\begin{defn}
\label{def-reg}
	For each $t \geq 0$, we let $(\rsm^{(t)}_i)_{i \geq 1}$ be independent uniform random samples in $S_t$ and denote $\P_t(\cdot): = \P(\cdot \mid \mathcal F_t)$, $\E_t[\cdot]: = \E[\cdot \mid \mathcal F_t]$. We say $S_t$ is regular if
	\begin{equation}
		\P_t\Big(|\langle \rsm^{(t)}_1,\rsm^{(t)}_2 \rangle| \leq C_2 \sqrt{N}\sqrt{|Q_t| + \log N} \Big) \geq 1 - N^{-10}\,.
	\end{equation}
\end{defn}
Roughly speaking, $S_t$ is regular if  two random samples from $S_t$ are almost orthogonal with high probability. Define the stopping time
\begin{equation}
	\tau_{S} := \inf \left\{  t \geq 0:	S_t \text{ is not regular}\right\}.
\end{equation}
The following lemma says that for regular $S_t$, $Y_{t+1}$ is roughly of order $N^{-1/2}$.
\begin{lemma}There exists a constant $C>0$ such that for all $t \geq 0$,
\label{Xttail}
\begin{equation}
	\1_{\tau_S >t}	\P_t\left(| Y_{t+1}|  \geq C\sqrt{\frac{|Q_t| + \log N }{N}}x\right) \leq C\exp (-x)\,.
\end{equation}
\end{lemma}
By \eqref{eq:QtXt}, Lemma \ref{Xttail} provides an upper bound on $Y_{t+1}$, which can be used to control the increment of $|Q_t|$. This will be one of the key ingredients in proving Lemma \ref{TheInduction}, which gives an upper bound on $|Q_t|$ for time $t$ before a stopping time defined below. 

With $C_3>0$ a constant to be determined later in Lemma \ref{ST-123}, we define
\begin{equation}
\label{def-tau123}
\begin{split}
&\tau_{Y} := \inf \left\{  t \geq 1:	| Y_t|  \geq C_3\sqrt{\frac{|Q_{t-1}| + \log N}{N}} \log N \right\}\,,\\
&\tau_{Q} := \inf \left\{  t \geq 0:	|Q_t| \geq (\log N)^2 \right\}\,,	
\end{split}
\end{equation}and
\begin{equation}
\label{def-tau}
	\tau = \tau_S \wedge \tau_Y \wedge \tau_Q\,.
\end{equation}

\begin{lemma}
\label{TheInduction}
There exists a constant $C>0$ such that for all $t \geq 1$,
	\begin{equation}
	\label{eq:QtIndBd}
		\E[|Q_{t \wedge \tau -1}|] \leq \exp(Ct/N) \log N\,.
	\end{equation}
\end{lemma}
Lemma \ref{timetau} says that the stopping time $\tau$ will occur later than time $\alpha N$ with high probability. Theorem \ref{Stub} will be a direct consequence of Lemmas \ref{TheInduction} and \ref{timetau}. 
\begin{lemma}Under Assumption \ref{assumption1}, for any $\alpha < \alpha_c$,
\label{timetau}
	$$\lim_{N \to \infty} \P(\tau < \alpha N) = 0\,.$$
\end{lemma}
 \begin{proof}[\bf Proof of Theorem \ref{Stub}]
By Lemmas \ref{TheInduction} and \ref{timetau}, the probability bound \eqref{eq:Stub} follows from Markov's inequality.
\end{proof}

\subsection{Proof of Lemma \ref{Xttail}}
In this section, we consider an arbitrary subset $A \subset \Sigma_N$ and define the the probability measure $\P_A$ under which $(\rsm_i)_{i \geq 1}$ are independent uniformly random vectors in $A$.

Lemma \ref{Xttail} follows from the following lemma.
\begin{lemma}
\label{LD}
There exists constant $C,c>0$ depending only on $\kappa$ such that for all $N$ sufficiently large, all $A\subset \Sigma_N$, and $\varphi_N \in \R$ satisfying 
	\begin{equation}
	\label{eq:deltaA-as}
		\delta_A : = \P_A(|\langle \rsm_1,\rsm_2 \rangle| > \varphi_N\sqrt{N}) \leq \varphi_N^{-2}N^{-2}\,,
	\end{equation}
 we have
	\begin{equation}
	\label{eq:LD}
		\P\left(\Big| \frac{|H_\kappa(X) \cap A|}{|A|} - p(\kappa) \Big|> \frac{\varphi_Nx}{\sqrt{N}}\right) \leq C \exp (-cx)\,,
	\end{equation}
	where $X$ is a standard $N$-dimensional Gaussian vector. 
\end{lemma}
The rest of this section is devoted to prove Lemma \ref{LD}.
\begin{lemma}
\label{kpt}
For all $N$ sufficiently large and every $B \subset A \subset \Sigma_N$, if we denote 
	$$\delta_A = \P_A(|\langle \rsm_1,\rsm_2 \rangle| >  \varphi_N\sqrt{N}) \quad \text{and } \quad q = \P_A(\rsm_1 \not \in B)  \,,$$
	 then for all $1 \leq k \leq \frac{q^2}{4\delta_A}$
\begin{equation}
	\P_A\big(E_k(B)\big) \geq \Big(q - \frac{2k \delta_A}{q} \Big)^k \,,
\end{equation}
where
	\begin{equation}
	\label{eq:def-Ek}
		E_k(B) := \big\{\rsm_i \not \in B, \forall 1 \leq i \leq k\big\} \cap\big\{ |\langle \rsm_i,\rsm_j\rangle| \leq \varphi_N\sqrt{N} \quad  \forall 1 \leq i,j \leq k, i \not = j\big\}.
	\end{equation}
\end{lemma}
\begin{proof} 

We will prove the following by induction
\begin{equation}
\label{eq:Ek-induction}
	\P_A(E_{t+1} \mid E_t) \geq q - \frac{2t \delta_A}{q} \quad \text{ for}\quad  1 \leq t \leq \frac{q^2}{4\delta_A}\,.
\end{equation}

First, since $\P_A(E_1) = \P_A(\rsm_1 \not \in B) =  1 - q$, we see that \eqref{eq:Ek-induction} holds for $t= 1$. 

Next, we suppose \eqref{eq:Ek-induction} is true for $t = k$ with $k \leq \frac{q^2}{4\delta_A}-1$. Since $\rsm_{k+1}$ is independent of $E_k$, we have
	\begin{align*}
		\P_A(E_{k+1} \mid E_k) &\geq  \P_A(\rsm_{k+1} \not \in B \mid E_k ) - \sum_{i = 1}^k \P_A(|\langle \rsm_{i},\rsm_{k+1}\rangle|> \varphi_N\sqrt{N} \mid E_k )\\
		& = q - k \P_A(|\langle \rsm_k,\rsm_{k+1}\rangle|> \varphi_N\sqrt{N} \mid E_k )\,.
	\end{align*}
Also, since $(\rsm_{k+1},\rsm_{k})$ is independent of $E_{k-1}$ and $E_{k-1} \supset E_k$,
	\begin{align*}
		\P_A(|\langle \rsm_k,\rsm_{k+1}\rangle| > \varphi_N\sqrt{N} \mid E_k ) \leq &\frac{\P_A(\{|\langle \rsm_k,\rsm_{k+1}\rangle| > \varphi_N\sqrt{N} \}\cap E_{k-1} )}{\P_A(E_k )} \\
		= &\P_A(|\langle \rsm_1,\rsm_{k+1}\rangle| > \varphi_N\sqrt{N}) \cdot \frac{\P(E_{k-1})}{\P(E_k)}\,.
	\end{align*}
By the induction hypothesis $$\frac{\P(E_{k-1})}{\P(E_k)} = \frac 1 {\P(E_k \mid E_{k-1})} \leq\frac q 2 \,,$$
where we used the assumption $k \leq \frac{q^2}{4\delta_A} -1$.  Combining with the previous two inequalities, we get $$\P_A(E_{k+1} \mid E_k)  \geq q - \frac{2k \delta_A}{q}\,.$$ Hence \eqref{eq:Ek-induction} holds.\end{proof}

\begin{lemma}
\label{gaussian}
	For any $\kappa >0$ and $m \in \N$, there exists constants $c_\kappa, C>0$ such that if we let $(X_i)_{i \geq 1}$ be a Gaussian vector with
	\begin{equation}
	\E[X_i] = 0, \quad	\E[X_i^2] = 1, \quad |E[X_iX_j]| \leq c_\kappa m^{-1}\,,
	\end{equation}
then 
\begin{align}
\label{eq:gaussian}
	&\P(\min_{1 \leq i \leq m} |X_i| > \kappa) \leq C\P\big(|X_1| > \kappa \big)^m \,,\\
	&\P(\max_{1 \leq i \leq m} |X_i| \leq \kappa) \leq C\P\big(|X_1| \leq \kappa \big)^m\,. \label{eq:gaussian-2}
\end{align}
 \end{lemma}
\begin{proof}
We first prove \eqref{eq:gaussian}. Let $\Sigma = (\Sigma_{ij}) = (\E[X_iX_j]) \in \R^{m \times m}$, then
	\begin{equation*}
	\begin{split}
				\P\left(\min_{1 \leq i \leq m}|X_i| > \kappa \right) &= \sqrt{\frac{1}{(2 \pi)^m |\Sigma|}} \int_{\max_{1 \leq i \leq m}|x_i| > \kappa} \exp\left\{ -\frac{x^\tran (\Sigma^{-1} - I)x + x^\tran x}{2}\right\} \dif \bx\\
		& = |\Sigma|^{-\frac 1 2} \E\left[ \exp\left\{ -\frac{Z^\tran (\Sigma^{-1} - I)Z}{2}\right\};\min_{1 \leq i \leq m}|Z_i| > \kappa\right]
	\end{split}
	\end{equation*}
	where $Z = (Z_1,...,Z_m)$ is a standard Gaussian vector. Since $(Z_1,...,Z_m)$ and $(e_1Z_1,...,e_mZ_m)$ has the same distribution for every $e_i \in \{-1,1\}$, we have 
		\begin{equation*}
		\P\left(\min_{1 \leq i \leq m}|X_i| > \kappa \right) = |\Sigma|^{-\frac 1 2} \E\left[ \sum_{(e_i)_{i=1}^m \in \{-1,1\}^m}\exp\left\{ -\frac 1 2 \sum_{1 \leq i,j \leq m} g_{ij} e_i e_j\right\};\min_{1 \leq i \leq m}|Z_i| > \kappa\right]
	\end{equation*}
	where $g_{ij} = \delta_{ij} Z_i Z_j$ and $\delta_{ij}$ is the $(i,j)$-th element of the matrix $\Sigma^{-1} - I$. Note that $\Sigma_{ii} = 1$ and $|\Sigma_{ij} | \leq c_\kappa m^{-1}$ for $i \not = j$ implies
	\begin{equation}
	\label{eq:sigma1}
	 	\max_{i,j} |\delta_{ij}| \leq C c_\kappa m^{-1}\,, \text{ and } ||\Sigma^{-1} - I||_2 \leq C c_\kappa\,.
	 \end{equation} 
Hence
	\begin{equation}
	\label{eq:gij-1}
		\sum_{i}g_{ii} \leq C c_\kappa m^{-1}|Z|_2^2, \quad \text{and} \quad \sum_{i,j} g_{ij}^2 \leq Cc_\kappa^2 m^{-2} |Z|_2^4\,.
	\end{equation}
Now, we first note that by \eqref{eq:sigma1}, for any $q >0$
\begin{equation*}
\begin{split}
		\E\left[ \exp\left\{ -\frac{Z^\tran (\Sigma^{-1} - I)Z}{2}\right\};|Z|_2^2 \geq q m\right] &\leq \E\left[ \exp\left\{ \frac{C c_\kappa|Z|_2^2}{2}\right\};|Z|_2^2 \geq q m\right]\\
		& = (1 - Cc_\kappa)^{\frac{m}{2}}\P (|Z|_2^2 \geq (1 - C c_\kappa) q m)\,.
\end{split}
\end{equation*}
Since the tail probability of a Chi-square random variable decays exponentially, we see that for $c_\kappa$ sufficiently small and $q$ sufficiently large,
\begin{equation}
\label{eq:Zlarge}
		\E\left[ \exp\left\{ -\frac{Z^\tran (\Sigma^{-1} - I)Z}{2}\right\};|Z|_2^2 \geq q m\right] \leq \P\big(|Z_1| > \kappa \big)^m\,.
\end{equation}
	Next, by \cite{MR279864}, we have that if $|\sum_{i}g_{ii}|$ and $\sum_{i,j} g_{ij}^2$ are sufficiently small, then
	\begin{equation}
		\sum_{(e_i)_{i=1}^m \in \{-1,1\}^m}\exp\left\{ -\frac 1 2 \sum_{1 \leq i,j \leq m} g_{ij} e_i e_j \right\} \leq C\,.
	\end{equation}
But we see from \eqref{eq:gij-1} that if $c_\kappa$ is sufficiently small, then on the event $\{|Z|^2_2 \leq q m\}$ the condition that $|\sum_{i}g_{ii}|$ and $\sum_{i,j} g_{ij}^2$ are sufficiently small holds. Combining with \eqref{eq:Zlarge}, we see that 
\begin{equation*}
	 \E\left[ \sum_{(e_i)_{i=1}^m \in \{-1,1\}^m}\exp\left\{ -\frac 1 2 \sum_{1 \leq i,j \leq m} g_{ij} e_i e_j\right\};\min_{1 \leq i \leq m}|Z_i| > \kappa\right] \leq C\P\big(|Z_1| > \kappa \big)^m\,.
\end{equation*}
Finally, \cite{ostrowski1938approximation} gives $|\Sigma| \geq 1- c_\kappa$. We complete the proof of \eqref{eq:gaussian}. 

The other inequality \eqref{eq:gaussian} can be proved verbatim by changing all occurrences of $\min_{1 \leq i \leq m} |X_i| > \kappa$ and $\min_{1 \leq i \leq m} |Z_i| > \kappa$ to $\max_{1 \leq i \leq m} |X_i| \leq \kappa$ and $\max_{1 \leq i \leq m} |Z_i| \leq \kappa$. In fact, it would be even easier as $\max_{1 \leq i \leq m} |Z_i| \leq \kappa$ implies $|Z|_2^2 \leq \kappa m$, and thus there is no need to consider the case $\{|Z|_2^2 \geq q m\}$.
\end{proof}

\begin{proof}[\bf Proof of Lemma \ref{LD}]
We will prove that
	\begin{equation}
	\label{eq:LD-1}
		\P\left( \frac{|H_\kappa(X) \cap A|}{|A|} < p(\kappa)- \frac{\varphi_Nx}{\sqrt{N}}\right) \leq C \exp (-cx)\,.
	\end{equation}
The other direction can be proved similarly.  

Recall $E_k(\cdot)$ as in \eqref{eq:def-Ek} and denote
	\begin{equation}
		A^\perp_k := \{(\sigma_i)_{i=1}^k  \in A^k:|\langle \sigma_i,\sigma_j\rangle| \leq \varphi_N\sqrt{N} \quad  \forall 1 \leq i,j \leq k, i \not = j\}\,.
	\end{equation}
Then
	\begin{equation}
		\E\big[\P_A(E_k(H_\kappa(X))) \big]  = \sum_{(\sigma_i)_{i=1}^k \in A^\perp_k} |A|^{-k}\P(\sigma_i \not \in H_\kappa(X), \forall 1 \leq i \leq k)\,.
	\end{equation}
By Lemma \ref{gaussian}, we have that for $k = \lfloor  c_\kappa \varphi_N^{-1} \sqrt{N} \rfloor$ and all $(\sigma_i)_{i=1}^k \in A^\perp_k$,
\begin{equation}
	\P(\sigma_i \not \in H_\kappa(X), \forall 1 \leq i \leq k) \leq C \P(|Z| > \kappa )^k\,.
\end{equation}
Hence
	\begin{equation}
	\label{eq:ExPAEH}
		\E\big[\P_A(E_k(H_\kappa(X))) \big]  \leq C \P(|Z| > \kappa )^k\,.
	\end{equation}
On the other hand, we set
\begin{equation}
	\Delta = \frac{\varphi_N x}{\sqrt{N}}\,.
\end{equation}
Then \eqref{eq:deltaA-as} implies $\Delta\geq \frac{4 k \delta_A}{q}$. Therefore, it follows from Lemma \ref{kpt} that on the event $\big\{\P_A\big(\rsm_1 \not \in H_\kappa(X) \mid H_\kappa(X)\big) \geq \P(|Z| > \kappa )  +  \Delta \big\}$, we have
\begin{align*}
	 \P_A\big(E_k[H_\kappa(X)] \mid H_\kappa(X)\big) &\geq \Big(\P_A(\rsm_1 \not \in H_\kappa(X) \mid H_\kappa(X)) -\Delta/2 \Big)^k \\
	 & \geq \P(|Z| > \kappa )^k e^{- k \Delta/2}\,.
\end{align*}
Combined with \eqref{eq:ExPAEH}, this yields \eqref{eq:LD-1}. 
\end{proof}

\subsection{Proof of Lemma \ref{TheInduction}}
By the definitions \eqref{def-tau123} and \eqref{def-tau}, we see that for all $0 \leq i \leq \tau-1$,
\begin{equation}
	|Y_i| \leq C\sqrt{\frac{|Q_{t-1}| + \log N}{N}} \leq C\sqrt{\frac{(\log N)^2 + \log N}{N}}\,.
\end{equation}
Hence it follows from $|\log (1 + x) - x| \leq  x^2$ for all $x \geq - 3/5$ that
\begin{equation}
	\Big|\sum_{i = 1}^{t \wedge \tau-1} \log (1 + Y_i)\Big| \leq  \Big|\sum_{i = 1}^{t \wedge \tau -1 } Y_i \Big| + \sum_{i = 1}^{t \wedge \tau -1 } Y_i^2 \leq \Big|\sum_{i = 1}^{t \wedge \tau } Y_i \Big| + \sum_{i = 1}^{t \wedge \tau } Y_i^2 + |Y_{t \wedge \tau}|\,.
\end{equation}
Since $\tau \leq \tau_S$, Lemma \ref{Xttail} implies
\begin{equation}
	\E[Y_i^2\1_{\tau \geq i}] = \E[\E_{i-1}[Y_i^2]\1_{\tau > i-1}] \leq C'\frac{\E[|Q_{i-1}|\1_{\tau > i-1}] + \log N}{N}\,.
\end{equation}
Therefore, since $(\sum_{i = 1}^{t \wedge \tau } Y_i)_{t \geq 1}$ is a martingale and $t\wedge \tau \leq \tau \leq \tau_S$, we have 
\begin{align*}
	\E\Big|\sum_{i = 1}^{t \wedge \tau-1} \log (1 + Y_i)\Big| &\leq  \sqrt{\E\sum_{i = 1}^{t \wedge \tau} Y_i^2} + \E\sum_{i = 1}^{t \wedge \tau} Y_i^2 + \E|Y_{t \wedge \tau}|\\
	& \leq 2\E\sum_{i = 1}^{t} Y_i^2\1_{ \tau \geq i} + 1 + \E|Y_{t \wedge \tau}|\\
	& \leq 2C' \cdot \frac{t\log N + \sum_{i = 1}^t \E[|Q_{i \wedge \tau - 1}|]}{N} + \frac{2}{p(\kappa)}\,,
\end{align*}
where $C'$ is a constant. Hence
\begin{equation}
	\E[|Q_{t\wedge \tau-1}|] \leq 3C' \Big( \frac{t\log N + \sum_{i = 1}^{t-1} \E[|Q_{i\wedge \tau-1}|]}{N} + \frac{1}{p(\kappa)}\Big)
\end{equation}
We claim that this yields \eqref{eq:QtIndBd}. To prove this, we define $b_1 := \E[|Q_{1\wedge \tau-1}|]$ and for $ t \geq 2$ define
\begin{equation}
\label{def:b-t}
	b_t := 3C' \Big( \frac{t\log N + \sum_{i = 1}^{t-1} b_i}{N} + \frac{1}{p(\kappa)}\Big)\,.
\end{equation}
Then a straightforward induction argument shows that
\begin{equation}
\label{eq:bt>E}
	b_t \geq \E[|Q_{t\wedge \tau-1}|] \quad \text{for }t \geq 1\,.
\end{equation}
On the other hand, \eqref{def:b-t} implies
\begin{align*}
	b_t &= 3C' \Big( \frac{\log N + b_{t-1}}{N} + \frac{(t-1)\log N + \sum_{i = 1}^{t-2} b_i}{N} + \frac{1}{p(\kappa)}\Big) \\
	&= 3C' \Big( \frac{\log N + b_{t-1}}{N} + \frac{b_{t-1}}{3C'}\Big)\,.
\end{align*}
This implies
\begin{equation}
\label{eq:bt-seq}
	b_t + \log N = \Big(1 + \frac{3C'}{N}\Big)(b_{t-1} + \log N) = \Big(1 + \frac{3C'}{N}\Big)^{t-1} (b_1 + \log N)\,.
\end{equation}
Note that $S_0 = \{-1,1\}^N$, $Q_0 = 0$ and $\P(\tau \geq 1) = 1$. We get $b_1 = 0$. Hence Lemma \eqref{eq:QtIndBd} follows from \eqref{eq:bt>E} and \eqref{eq:bt-seq}. We complete the proof of Lemma \ref{TheInduction}.

\subsection{Proof of Lemma \ref{timetau}}
\begin{lemma}
\label{Exp-R}
Under Assumption \ref{assumption1}, there exists a constant $C_2$ such that for every $1 \leq t \leq \alpha N$,
\begin{equation}
	\P \left( \P_t\Big(|\langle \rsm^{(t)}_1,\rsm^{(t)}_2 \rangle| \geq C_2 \sqrt{N}\sqrt{Q_t + \log N} \Big) \geq N^{-10} \right) \leq N^{-10}\,.
\end{equation}
As a result,
\begin{equation}
	\P(\tau_S\leq \alpha N) \leq N^{-10}\,. \label{eq:ST-1}
\end{equation}

\end{lemma}
\begin{proof}
	Since there exist a constant $c>0$ such that for any $\lambda > 0$ and $t \geq 0$,
\begin{equation}
	\E[|\{(\sigma_1,\sigma_2) \in S_t : |\langle \sigma_1,\sigma_2 \rangle| \geq \lambda \sqrt{N} \}|] \leq \exp(-c\lambda^2)\E[|S_t|]^2\,,
\end{equation}
we have that 
\begin{equation}
	\E \left[\P_t(|\langle \rsm^{(t)}_1,\rsm^{(t)}_2 \rangle| \geq \lambda \sqrt{N} ) \frac{|S_t|^2}{\E[|S_t|]^2} \right] \leq C\exp(-c\lambda^2)\,,
\end{equation}
where we used $\E[|S_t|^2] \leq C\E[|S_t|]^2$ (see \cite{aubin2019storage}).
Therefore,
\begin{align*}
	&\E \left[\exp\left(\frac{c}{2}\frac{\langle \rsm^{(t)}_1,\rsm^{(t)}_2 \rangle^2}{N}  + 2 Q_t\right) \right]\\
	 & \quad = \E \left[\int_{0}^\infty c\lambda e^{\frac{c \lambda^2}{2}}\P_t\Big(\frac{\langle \rsm^{(t)}_1,\rsm^{(t)}_2 \rangle^2}{N} \geq \lambda\Big)\dif \lambda \cdot \frac{|S_t|^2}{\E[|S_t|]^2}\right] \\
	& \quad  \leq  C \int_{0}^\infty c\lambda e^{\frac{-c \lambda^2}{2}}\dif \lambda\\
	& \quad  \leq C\,. 
\end{align*}
Hence, there exists a constant $C$ such that
\begin{equation}
	\P(|\langle \rsm^{(t)}_1,\rsm^{(t)}_2 \rangle| \geq C \sqrt{N}\sqrt{Q_t + \log N} ) \leq N^{-100}\,,
\end{equation}
which yields the desired result.
\end{proof}

 \begin{lemma}There exists a constant $C_3>0$ such that
 \label{ST-123}
\begin{align}
	&\P(\tau_Y \leq \tau_S, \tau_Y \leq \alpha N) \leq N^{-8}\label{eq:ST-2}\\
	&\P(|Q_{\tau_Q -1}| \leq (\log N)^2/2 , \tau_Q \leq \tau_S, \tau_Q \leq \alpha N)  \leq N^{-8}\label{eq:ST-3}
\end{align}	
\end{lemma}
\begin{proof}
\eqref{Xttail} follows from Lemma \ref{Xttail} and a union bound. Lemma \ref{Xttail} gives	\begin{equation*}
	\begin{split}
		&\P(|Q_{\tau_Q -1}| \leq (\log N)^2/2 , \tau_Q \leq \tau_S, \tau_Q =t+1) \\
		& \quad = \E\big[\P_t(|Q_t + \log (1 + Y_{t+1})| \geq (\log N)^2) \1_{|Q_t| \leq (\log N)^2/2,\tau_S \geq t +1}  \big]\\
		& \quad \leq  \E\big[\P_t(Y_t \leq - 1/2) \1_{|Q_t| \leq (\log N)^2/2,\tau_S \geq t +1}  \big]\\
		& \quad \leq N^{-10}\,.
	\end{split}
	\end{equation*}
This yields \eqref{eq:ST-3}.
\end{proof}

\begin{proof}[\bf Proof of Lemma \ref{timetau}]
	Write \begin{align*}
	\P(\tau \geq t) &\leq \P(\tau_S\leq \alpha N) + \P(\tau_Y \leq \tau_S, \tau_Y \leq \alpha N) + \P(\tau_Q \leq t,\tau_Q = \tau,|Q_{\tau_Q -1}| \leq (\log N)^2/2)\\
	& \quad + \P(\tau_Q \leq t,\tau_Q = \tau,|Q_{\tau_Q -1}| \geq (\log N)^2/2)\,.
\end{align*}
Lemma \ref{ST-123} and \eqref{eq:ST-1} give upper bounds on the first three probability on the right hand. In addition, Lemma \ref{TheInduction} and the Markov inequality yield
\begin{align*}
	&\P(\tau_Q \leq t,\tau_Q = \tau,|Q_{\tau_Q -1}| \geq (\log N)^2/2)\\
	& \quad  \leq ((\log N)^2/2)^{-1}\E[Q_{t \wedge \tau - 1} ;\tau_Q \leq t,\tau_Q = \tau,|Q_{\tau_Q -1}| \geq (\log N)^2/2] \\
	& \quad \leq ((\log N)^2/2)^{-1}\E[|Q_{t \wedge \tau -1}|]\\
	& \quad \leq (\log N)^{-1/2}\,.
\end{align*}
Combining these bounds yields $\P(\tau \geq t) \leq 2(\log N)^{-1/2}$.
\end{proof}

\section{From the planted model to the random model}
\label{secQuiet}

Theorem \ref{lem:strong-clustering} will follow from  Lemma \ref{pl-cluster} and the following lemma on the planted model.
\begin{lemma}
\label{contiguity}
Suppose 
$A \subseteq \{ (\sigma, U) : \sigma \in U \subseteq \Sigma_N \}$, 
and suppose $\alpha < \alpha_c$.  If $\P_{\pl}((\sigma^*, S(\mathbf X)) \in A) \leq N^{-m_N}$ for some $m_N \to \infty$, then $$\P_{\mathrm{r}}((\sigma^*, S(\mathbf X)) \in A) \to 0 \quad\text{ as }N \to \infty\,.$$
\end{lemma}

We start by characterizing the planted distribution. 
\begin{lemma}
\label{17}For any $N,m \geq 1$, with $A \subseteq \{ (\sigma, U) : \sigma \in U \subseteq \Sigma_N \}$, 
\label{3.1}
		\begin{equation}
		\P_{\pl}((\sigma^*, S(\mathbf X)) \in A) = \E _{\mathrm{r}}\left[\1_{(\sigma_N, S(\mathbf X)) \in A}\cdot \frac{|S(\mathbf X)|}{\E|S(\mathbf X)|}\right] \cdot \P ( S(\mathbf X) \ne \emptyset )\,.
	\end{equation}
\end{lemma}
\begin{proof}
	It suffices to prove that for any $\sigma \in \Sigma_N$ and $U \subset \Sigma_N$ such that $\sigma \in U$, we have
	\begin{equation}
	 \label{eq:71}
		\P_{\pl}(\sigma^* = \sigma, S (\mathbf X) = U) = \P_{\mathrm{r}}(\sigma^* = \sigma, S(\mathbf X) = U) \frac{|U|}{\E|S(\mathbf X)|} \cdot \P ( S(\mathbf X) \ne \emptyset )\,.
	\end{equation}
To this end, write
\begin{equation*}
\begin{split}
		\P_{\mathrm{pl}}(\sigma^* = \sigma, S (\mathbf X) = U) &= |\Sigma_N|^{-1}\P(S (\mathbf X) = U \mid \sigma \in S (\mathbf X))\\
		& = \frac{\P(S (\mathbf X) = U ,\sigma \in S (\mathbf X))}{|\Sigma_N|\P(\sigma \in S (\mathbf X))}\\
		& = \frac{\P(S (\mathbf X) = U)}{|\Sigma_N|\P(\sigma \in S (\mathbf X))}\,.
\end{split}
\end{equation*}
Note that $\E|S(\mathbf X)| = |\Sigma_N|\P(\sigma \in S (\mathbf X))$ and that $$ \P_{\mathrm{r}}(\sigma^* = \sigma, S(\mathbf X) = U) \cdot \P ( S(\mathbf X) \ne \emptyset ) = |U|^{-1}\P(S(\mathbf X) = U)\,.$$
We get \eqref{eq:71}, and thus complete the proof of Lemma \ref{17}.
\end{proof}

\begin{proof}[\bf Proof of Lemma \ref{contiguity}]
Theorem \ref{PartitionF} implies the event
	\begin{equation}
		G =\left\{\frac{|S|}{\E|S|}  \geq \exp \{-m_N \log N\}\right\}
	\end{equation}
	has probability  $1-o(1)$. Therefore, as $N \to \infty$,
	\begin{equation}
	\begin{split}
		\P_{\mathrm r}((\sigma^*, S) \in A) &\leq \P_{\mathrm r}\Big(\{(\sigma^*, S) \in A\} \cap  G\Big) + \P(G)/\P(S \not = \emptyset)\\
		& \leq \exp \{m_N \log N\}\E\left[\1_{(\sigma^*, S) \in A}\cdot \frac{|S|}{\E|S|}\right] + o(1)\\
		& = \exp\{m_N \log N\}\P_{\mathrm{pl}}((\sigma^*, S) \in A)/\P(S \not = \emptyset) + o(1)\\
		& = o(1)\,.
	\end{split} 
	\end{equation}
where we used Lemma \ref{3.1} and Theorem \ref{PartitionF}. This yields Lemma \ref{contiguity}.
\end{proof}

Finally we prove Theorem~\ref{lem:strong-clustering}.
\begin{proof}[Proof of Theorem~\ref{lem:strong-clustering}] Let 
\begin{equation}
	A = \left\{(\tau, U): \{ \sigma \in U: \langle \sigma, \tau \rangle  \geq (d_{c} + \delta)N \} \not = \{\tau\} \right\}\,.
\end{equation}
Then Lemma \ref{pl-cluster} implies that
\begin{equation}
	\P_{\pl}((\sigma^*, S(\mathbf X )) \in A) \leq \exp\left\{- c\sqrt{N}\right\}
\end{equation}
Combined with Lemma \ref{contiguity}, this yields
$$ \P_{\mathrm{r}}((\sigma^*, S(\mathbf X)) \in A) = o(1)   \,,$$
and thus Theorem \ref{lem:strong-clustering}.
\end{proof}

\section*{Acknowledgements}
We thank Benjamin Aubin and Lenka Zdeborov{\'a} for inspiring discussions and introducing us to this problem.  

\newcommand{\etalchar}[1]{$^{#1}$}


\begin{thebibliography}{KMRT{\etalchar{+}}07}

\bibitem[Abb17]{abbe2017community}
Emmanuel Abbe.
\newblock Community detection and stochastic block models: recent developments.
\newblock {\em The Journal of Machine Learning Research}, 18(1):6446--6531,
  2017.

\bibitem[ACO08]{achlioptas2008algorithmic}
Dimitris Achlioptas and Amin Coja-Oghlan.
\newblock Algorithmic barriers from phase transitions.
\newblock In {\em 2008 49th Annual IEEE Symposium on Foundations of Computer
  Science}, pages 793--802. IEEE, 2008.

\bibitem[ACORT11]{achlioptas2011solution}
Dimitris Achlioptas, Amin Coja-Oghlan, and Federico Ricci-Tersenghi.
\newblock On the solution-space geometry of random constraint satisfaction
  problems.
\newblock {\em Random Structures \& Algorithms}, 38(3):251--268, 2011.

\bibitem[AM02]{achlioptas2002asymptotic}
Dimitris Achlioptas and Cristopher Moore.
\newblock The asymptotic order of the random k-{SAT} threshold.
\newblock In {\em The 43rd Annual IEEE Symposium on Foundations of Computer
  Science, 2002. Proceedings.}, pages 779--788. IEEE, 2002.

\bibitem[APZ19]{aubin2019storage}
Benjamin Aubin, Will Perkins, and Lenka Zdeborov{\'a}.
\newblock Storage capacity in symmetric binary perceptrons.
\newblock {\em Journal of Physics A: Mathematical and Theoretical}, 2019.

\bibitem[Bal09]{baldassi2009generalization}
Carlo Baldassi.
\newblock Generalization learning in a perceptron with binary synapses.
\newblock {\em Journal of Statistical Physics}, 136(5):902--916, 2009.

\bibitem[BBBZ07]{baldassi2007efficient}
Carlo Baldassi, Alfredo Braunstein, Nicolas Brunel, and Riccardo Zecchina.
\newblock Efficient supervised learning in networks with binary synapses.
\newblock {\em Proceedings of the National Academy of Sciences},
  104(26):11079--11084, 2007.

\bibitem[BBC{\etalchar{+}}16]{baldassi2016unreasonable}
Carlo Baldassi, Christian Borgs, Jennifer~T Chayes, Alessandro Ingrosso, Carlo
  Lucibello, Luca Saglietti, and Riccardo Zecchina.
\newblock Unreasonable effectiveness of learning neural networks: From
  accessible states and robust ensembles to basic algorithmic schemes.
\newblock {\em Proceedings of the National Academy of Sciences},
  113(48):E7655--E7662, 2016.

\bibitem[BCOE17]{bapst2017planting}
Victor Bapst, Amin Coja-Oghlan, and Charilaos Efthymiou.
\newblock Planting colourings silently.
\newblock {\em Combinatorics, probability and computing}, 26(3):338--366, 2017.

\bibitem[BCOH{\etalchar{+}}16]{bapst2016condensation}
Victor Bapst, Amin Coja-Oghlan, Samuel Hetterich, Felicia Ra{\ss}mann, and Dan
  Vilenchik.
\newblock The condensation phase transition in random graph coloring.
\newblock {\em Communications in Mathematical Physics}, 341(2):543--606, 2016.

\bibitem[BDVLZ20]{baldassi2019clustering}
Carlo Baldassi, Riccardo Della~Vecchia, Carlo Lucibello, and Riccardo Zecchina.
\newblock Clustering of solutions in the symmetric binary perceptron.
\newblock {\em Journal of Statistical Mechanics: Theory and Experiment},
  2020(7):073303, 2020.

\bibitem[BIL{\etalchar{+}}15]{baldassi2015subdominant}
Carlo Baldassi, Alessandro Ingrosso, Carlo Lucibello, Luca Saglietti, and
  Riccardo Zecchina.
\newblock Subdominant dense clusters allow for simple learning and high
  computational performance in neural networks with discrete synapses.
\newblock {\em Physical review letters}, 115(12):128101, 2015.

\bibitem[BIL{\etalchar{+}}16]{baldassi2016local}
Carlo Baldassi, Alessandro Ingrosso, Carlo Lucibello, Luca Saglietti, and
  Riccardo Zecchina.
\newblock Local entropy as a measure for sampling solutions in constraint
  satisfaction problems.
\newblock {\em Journal of Statistical Mechanics: Theory and Experiment},
  2016(2):023301, 2016.

\bibitem[BSZ19]{bartha2019breaking}
Zsolt Bartha, Nike Sun, and Yumeng Zhang.
\newblock Breaking of {1RSB} in random regular {MAX-NAE-SAT}.
\newblock In {\em 2019 IEEE 60th Annual Symposium on Foundations of Computer
  Science (FOCS)}, pages 1405--1416. IEEE, 2019.

\bibitem[BZ06]{braunstein2006learning}
Alfredo Braunstein and Riccardo Zecchina.
\newblock Learning by message passing in networks of discrete synapses.
\newblock {\em Physical review letters}, 96(3):030201, 2006.

\bibitem[COEJ{\etalchar{+}}18]{coja2018charting}
Amin Coja-Oghlan, Charilaos Efthymiou, Nor Jaafari, Mihyun Kang, and Tobias
  Kapetanopoulos.
\newblock Charting the replica symmetric phase.
\newblock {\em Communications in Mathematical Physics}, 359(2):603--698, 2018.

\bibitem[COKPZ18]{coja2018information}
Amin Coja-Oghlan, Florent Krzakala, Will Perkins, and Lenka Zdeborov{\'a}.
\newblock Information-theoretic thresholds from the cavity method.
\newblock {\em Advances in Mathematics}, 333:694--795, 2018.

\bibitem[Cov65]{cover1965geometrical}
Thomas~M Cover.
\newblock Geometrical and statistical properties of systems of linear
  inequalities with applications in pattern recognition.
\newblock {\em IEEE Transactions on Electronic Computers}, (3):326--334, 1965.

\bibitem[COZ12]{coja2012condensation}
Amin Coja-Oghlan and Lenka Zdeborov{\'a}.
\newblock The condensation transition in random hypergraph 2-coloring.
\newblock In {\em Proceedings of the twenty-third annual ACM-SIAM symposium on
  Discrete Algorithms}, pages 241--250. SIAM, 2012.

\bibitem[DS19]{ding2019capacity}
Jian Ding and Nike Sun.
\newblock Capacity lower bound for the {I}sing perceptron.
\newblock In {\em Proceedings of the 51st Annual ACM SIGACT Symposium on Theory
  of Computing}, pages 816--827, 2019.

\bibitem[DSS14]{ding2014satisfiability}
Jian Ding, Allan Sly, and Nike Sun.
\newblock Satisfiability threshold for random regular {NAE-SAT}.
\newblock In {\em Proceedings of the forty-sixth annual ACM symposium on Theory
  of computing}, pages 814--822, 2014.

\bibitem[Fri99]{friedgut1999sharp}
Ehud Friedgut.
\newblock Sharp thresholds of graph properties, and the k--sat problem.
\newblock {\em Journal of the American mathematical Society}, 12(4):1017--1054,
  1999.

\bibitem[Gar87]{G87}
Elizabeth Gardner.
\newblock Maximum storage capacity in neural networks.
\newblock {\em EPL (Europhysics Letters)}, 4(4):481, 1987.

\bibitem[GD88]{GD88}
E~Gardner and B~Derrida.
\newblock Optimal storage properties of neural network models.
\newblock {\em Journal of Physics A: Mathematical and general}, 21(1):271,
  1988.

\bibitem[HK14]{huang2014origin}
Haiping Huang and Yoshiyuki Kabashima.
\newblock Origin of the computational hardness for learning with binary
  synapses.
\newblock {\em Physical Review E}, 90(5):052813, 2014.

\bibitem[HW71]{MR279864}
D.~L. Hanson and F.~T. Wright.
\newblock A bound on tail probabilities for quadratic forms in independent
  random variables.
\newblock {\em Ann. Math. Statist.}, 42:1079--1083, 1971.

\bibitem[HWK13]{huang2013entropy}
Haiping Huang, KY~Michael Wong, and Yoshiyuki Kabashima.
\newblock Entropy landscape of solutions in the binary perceptron problem.
\newblock {\em Journal of Physics A: Mathematical and Theoretical},
  46(37):375002, 2013.

\bibitem[KM89]{KM89}
Werner Krauth and Marc M{\'e}zard.
\newblock Storage capacity of memory networks with binary couplings.
\newblock {\em Journal de Physique}, 50(20):3057--3066, 1989.

\bibitem[KMRT{\etalchar{+}}07]{krzakala2007gibbs}
Florent Krzaka{\l}a, Andrea Montanari, Federico Ricci-Tersenghi, Guilhem
  Semerjian, and Lenka Zdeborov{\'a}.
\newblock Gibbs states and the set of solutions of random constraint
  satisfaction problems.
\newblock {\em Proceedings of the National Academy of Sciences},
  104(25):10318--10323, 2007.

\bibitem[KR98]{kim1998covering}
Jeong~Han Kim and James~R Roche.
\newblock Covering cubes by random half cubes, with applications to binary
  neural networks.
\newblock {\em Journal of Computer and System Sciences}, 56(2):223--252, 1998.

\bibitem[KZ09]{krzakala2009hiding}
Florent Krzakala and Lenka Zdeborov{\'a}.
\newblock Hiding quiet solutions in random constraint satisfaction problems.
\newblock {\em Physical review letters}, 102(23):238701, 2009.

\bibitem[Mol18]{molloy2018freezing}
Michael Molloy.
\newblock The freezing threshold for k-colourings of a random graph.
\newblock {\em Journal of the ACM (JACM)}, 65(2):1--62, 2018.

\bibitem[MRT11]{montanari2011reconstruction}
Andrea Montanari, Ricardo Restrepo, and Prasad Tetali.
\newblock Reconstruction and clustering in random constraint satisfaction
  problems.
\newblock {\em SIAM Journal on Discrete Mathematics}, 25(2):771--808, 2011.

\bibitem[MSL92]{mitchell1992hard}
David Mitchell, Bart Selman, and Hector Levesque.
\newblock Hard and easy distributions of {SAT} problems.
\newblock In {\em AAAI}, volume~92, pages 459--465, 1992.

\bibitem[Ost38]{ostrowski1938approximation}
Alexander~Markowitsch Ostrowski.
\newblock Sur l'approximation du determinant de fredholm par les determinants
  des syst{\`e}mes d'{\'e}quations lin{\'e}aires.
\newblock {\em Ark. Math. Stockholm}, 26A:1--15, 1938.

\bibitem[SSZ16]{sly2016number}
Allan Sly, Nike Sun, and Yumeng Zhang.
\newblock The number of solutions for random regular {NAE-SAT}.
\newblock In {\em 2016 IEEE 57th Annual Symposium on Foundations of Computer
  Science (FOCS)}, pages 724--731. IEEE, 2016.

\bibitem[STS90]{sompolinsky1990learning}
Haim Sompolinsky, Naftali Tishby, and H~Sebastian Seung.
\newblock Learning from examples in large neural networks.
\newblock {\em Physical Review Letters}, 65(13):1683, 1990.

\bibitem[Tal99]{talagrand1999intersecting}
Michel Talagrand.
\newblock Intersecting random half cubes.
\newblock {\em Random Structures \& Algorithms}, 15(3-4):436--449, 1999.

\bibitem[Tal10]{talagrand2010mean}
Michel Talagrand.
\newblock {\em Mean field models for spin glasses: Volume I: Basic examples},
  volume~54.
\newblock Springer Science \& Business Media, 2010.

\bibitem[TAP77]{thouless1977solution}
David~J Thouless, Philip~W Anderson, and Robert~G Palmer.
\newblock Solution of '{S}olvable model of a spin glass'.
\newblock {\em Philosophical Magazine}, 35(3):593--601, 1977.

\bibitem[Xu19]{xu2019sharp}
Changji Xu.
\newblock Sharp threshold for the {I}sing perceptron model.
\newblock {\em arXiv preprint arXiv:1905.05978}, 2019.

\bibitem[ZK07]{zdeborova2007phase}
Lenka Zdeborov{\'a} and Florent Krzaka{\l}a.
\newblock Phase transitions in the coloring of random graphs.
\newblock {\em Physical Review E}, 76(3):031131, 2007.

\bibitem[ZK16]{zdeborova2016statistical}
Lenka Zdeborov{\'a} and Florent Krzakala.
\newblock Statistical physics of inference: Thresholds and algorithms.
\newblock {\em Advances in Physics}, 65(5):453--552, 2016.

\bibitem[ZM08]{zdeborova2008constraint}
Lenka Zdeborov{\'a} and Marc M{\'e}zard.
\newblock Constraint satisfaction problems with isolated solutions are hard.
\newblock {\em Journal of Statistical Mechanics: Theory and Experiment},
  2008(12):P12004, 2008.

\end{thebibliography}
\end{document}